\documentclass[11pt]{article}
\usepackage{graphics}
\usepackage[dvips]{graphicx}

\usepackage{amsmath,amssymb,amsthm}
\usepackage[utf8]{inputenc}
\usepackage{tikz}
\usepackage{subfigure}
\usepackage{color}
\usepackage{comment}
\usepackage{float}
\usepackage[a4paper,margin=2.5cm]{geometry}

\title{
Periodic points in complex continued fractions  by J. Hurwitz
}

\author{Shin-ichi Yasutomi}

\date{}

\newtheorem{thm}{Theorem}[section]
\newtheorem{cor}[thm]{Corollary}
\newtheorem{lem}[thm]{Lemma}

\theoremstyle{definition}
\newtheorem{defn}{Definition}[section]
\newtheorem{example}{Example}[section]

\theoremstyle{remark}
\newtheorem{rem}{Remark}[section]

\numberwithin{equation}{section}

\usepackage{color}

\begin{document}

\maketitle


\begin{abstract}
J. Hurwitz \cite{JH} introduced an algorithm that generates a continued fraction expansion 
for complex numbers $\alpha \in \mathbb{C}$, 
where the partial quotients  belong to $(1+i)\mathbb{Z}[i]$.
 J. Hurwitz's work also provides a result analogous to Lagrange's theorem on periodic continued fractions, describing purely periodic points using the dual continued fraction expansion.
 S. Tanaka \cite{ST} examined the identical algorithm and
constructed the natural extension of the transformation generating the continued fraction expansion and established its ergodic properties.
In this paper, we aim to describe J. Hurwitz's insights into purely periodic points explicitly through the natural extension.
\end{abstract}

\section{Introduction}
J. Hurwitz \cite{JH} considered an algorithm that provides the following continued fraction expansion for $\alpha \in \mathbb{C}$:
where $a_j \in (1+i)\mathbb{Z}[i]$ for $j=0,1,\ldots$,
\begin{align}
\label{alpha=a_0}
\alpha = a_0 - \cfrac{1}{a_1 - \cfrac{1}{a_2 - \cfrac{1}{a_3 - \cfrac{1}{\ldots}}}}.
\end{align}
S. Tanaka \cite{ST} also examined an algorithm that yields the following continued fraction expansion for $\alpha \in \mathbb{C}$:
where $b_j \in (1+i)\mathbb{Z}[i]$ for $j=0,1,\ldots$,
\begin{align}
\label{alpha=b_0}
\alpha = b_0 + \cfrac{1}{b_1 + \cfrac{1}{b_2 + \cfrac{1}{b_3 + \cfrac{1}{\ldots}}}}.
\end{align}
Since for $\alpha$ in (\ref{alpha=a_0}), we have
\[
i\alpha = ia_0 + \cfrac{1}{ia_1 + \cfrac{1}{ia_2 + \cfrac{1}{ia_3 + \cfrac{1}{\ldots}}}},
\]
holds, the two algorithms are essentially the same, aside from minor differences.
Regarding J. Hurwitz's continued fraction algorithm, its historical connections to A. Hurwitz's continued fraction algorithm and S. Tanaka's continued fraction algorithm are traced in the literature \cite{O}, \cite{OJ}.
J. Hurwitz \cite{JH} provided a result analogous to Lagrange's theorem on periodic continued fractions in the case of regular continued fractions and described purely periodic points using the dual continued fraction expansion algorithm.
S. Tanaka \cite{ST} constructed the natural extension of the transformation that generates the continued fraction expansion and established its ergodic properties.
H. Nakada \cite{NA} also provided a different form of natural extension.
The description of purely periodic points given by J. Hurwitz is qualitative, and it can be explicitly expressed using this natural extension.
 This is the objective of this paper.
To facilitate a comparison between S. Tanaka's \cite{ST} and J. Hurwitz's \cite{JH} algorithms, we will reconstruct J. Hurwitz's algorithm in the form of the continued fraction expansion given by (\ref{alpha=b_0}).
The algorithms of J. Hurwitz and S. Tanaka are largely equivalent, but there are slight differences in the domains of their respective associated transformations.
We also provide the purely periodic points for Tanaka's algorithm.
Continued fraction expansions using imaginary quadratic fields as partial quotients
 have been considered since A. Hurwitz's continued fraction expansion \cite{AH}.
We refer the reader to \cite{G} and \cite{NNE}.
 Regarding the purely periodic points in A. Hurwitz's continued fraction expansion \cite{AH},
 G. González Robert \cite{G} has obtained a necessary condition and a sufficient condition, respectively.
The paper is organized as follows.
In Section 2, we define the continued fraction algorithms given by J. Hurwitz and  S. Tanaka.
We also describe the dual algorithm given by J. Hurwitz.
In Section 3, we present J. Hurwitz's results concerning the periodic points of the algorithm. 
In Section 4, we describe the natural extension given by H. Nakada. 
In Section 5, we present the necessary lemmas and offer some considerations regarding quadratic extensions over $\mathbb{Q}(i)$.
In Section 6, we state main theorems.
For a natural number $n$, the regular continued fraction of $\sqrt{n}$ becomes purely periodic from the second term onwards, and we observe that the same holds for the continued fractions of J. Hurwitz.

\section{Algorithms}
\begin{figure}
    \centering
\begin{tikzpicture}
\draw (1,0)--(0,1)--(-1,0)--(0,-1)--(1,0)--cycle;
\draw (-2,0)--(2,0);
\draw (0,-2)--(0,2);

\draw (-1,0.2)node[left]{$-1$};
\draw (1,0.2)node[right]{$1$};
\draw (-0.05,0.17)node[right]{$0$};
\draw (0,1.2)node[right]{$i$};
\draw (0,-1.2)node[right]{$-i$};

\coordinate (P) at (0,1);
\fill (P) circle [radius=1.5pt];
\coordinate (Q) at (0,-1);
\fill (Q) circle [radius=1.5pt];
\coordinate (R) at (-1,0);
\fill (R) circle [radius=1.5pt];
\coordinate (S) at (1,0);
\fill (R) circle [radius=1.5pt];
\coordinate (O) at (0,0);
\fill (O) circle [radius=1.5pt];
\coordinate (T) at (1,0);
\fill (T) circle [radius=1.5pt];

 \end{tikzpicture}
    \caption{$\overline{X}$. }
    \label{fig:jh-1}
\end{figure}
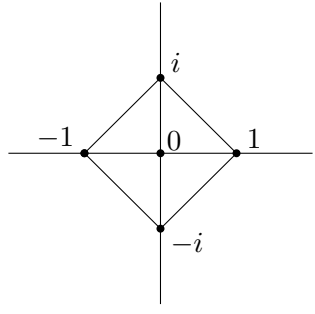

First, we outline J. Hurwitz's's algorithm.

Let $X=\{u(1+i)+v(1-i)|-1/2\leq u,v < 1/2\}$.

For $0 \leq k < 4 $, we define:

\begin{align*}
&L(i^k(1+i)) := i^k(1+i)\mathbb{Z}_{>0},\\
&L(i^k2) := \{i^k(m(1+i) + l(1-i)) \mid m, l\in \mathbb{Z}_{>0}\}.
\end{align*}

Then, the set $(1+i)Z[i] \setminus \{0\}$ is the disjoint union of the following:

\[
(1+i)Z[i] \setminus \{0\} = \bigcup_{0 \leq k < 4} L(i^k(1+i)) \cup \bigcup_{0 \leq k < 4} L(i^k2).
\]

For $w \in (1+i)\mathbb{Z}[i]$, the quadrilateral $Q_w$ is defined as follows (see Figure \ref{fig:jh-2}):
\begin{itemize}
    \item $Q_0 = \overline{X}$, where $\overline{X}$ denotes the closure of $X$.
    \item When $w \in L(i^k(1+i))$ for $0 \leq k < 4$,
    \begin{align*}
    &Q_w := (X^o + w) \cup \{ ui^k(1+i) - \frac{1}{2}i^k(1-i) + w \mid -\frac{1}{2} < u \leq \frac{1}{2} \} 
    \cup \\
    &\cup\left\{ \frac{1}{2}i^k(1+i) + ui^k(1-i) + w \mid -\frac{1}{2} \leq u \leq \frac{1}{2} \right\} \cup\\
    &\cup \{ ui^k(1+i) + \frac{1}{2}i^k(1-i) + w \mid -\frac{1}{2} < u \leq \frac{1}{2} \},
    \end{align*}
    where $X^o$ is the interior of $X$.
    \item When $w \in L(i^k2)$ for $0 \leq k < 4$,
    \begin{align*}
    &Q_w := (X^o + w) \cup \{ ui^k(1+i) + \frac{1}{2}i^k(1-i) + w \mid -\frac{1}{2} < u \leq \frac{1}{2} \}\cup\\ 
    &\cup \left\{ \frac{1}{2}i^k(1+i) + ui^k(1-i) + w \mid -\frac{1}{2} < u \leq \frac{1}{2} \right\}.
    \end{align*}
\end{itemize}
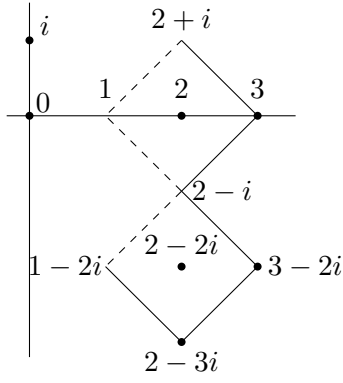
\begin{figure}
    \centering
\begin{tikzpicture}
\draw (2,-1)--(3,0)--(2,1);
\draw[dashed] (2,-1)--(1,0)--(2,1);
\draw (1,-2)--(2,-3)--(3,-2)--(2,-1);
\draw[dashed] (2,-1)--(1,-2);

\draw (-0.3,0)--(3.5,0);
\draw (0,-3.2)--(0,1.5);

\draw (3,0.1)node[above]{$3$};
\draw (1,0.1)node[above]{$1$};
\draw (-0.05,0.17)node[right]{$0$};
\draw (0,1.2)node[right]{$i$};
\draw (2,1)node[above]{$2+i$};
\draw (2,0.1)node[above]{$2$};
\draw (2,-2)node[above]{$2-2i$};
\draw (3,-2)node[right]{$3-2i$};
\draw (1.1,-2)node[left]{$1-2i$};
\draw (2,-3)node[below]{$2-3i$};
\draw (2,-1)node[right]{$2-i$};

\coordinate (P) at (3,0);
\fill (P) circle [radius=1.5pt];
\coordinate (Q) at (3,-2);
\fill (Q) circle [radius=1.5pt];
\coordinate (R) at (2,-3);
\fill (R) circle [radius=1.5pt];
\coordinate (O) at (0,0);
\fill (O) circle [radius=1.5pt];
\coordinate (O) at (2,0);
\fill (O) circle [radius=1.5pt];
\coordinate (O) at (2,-2);
\fill (O) circle [radius=1.5pt];
\coordinate (O) at (0,1);
\fill (O) circle [radius=1.5pt];

 \end{tikzpicture}
    \caption{$Q_2$ and $Q_{2-2i}$. }
    \label{fig:jh-2}
\end{figure}

Then, the set $\mathbb{C}$ is the disjoint union:
\begin{align*}
\mathbb{C}=\bigsqcup_{w\in (1+i)Z[i]} Q_w.
\end{align*}

For $z \in \mathbb{C}\backslash\mathbb{Z}[i]$, there exists a unique $w \in (1+i)\mathbb{Z}[i] $ such that $z \in Q_w $ and, we define $\lfloor z \rfloor_H = w$.
Additionally, when $z \in \mathbb{Z}[i] $, we define $\lfloor z \rfloor_H = z$.

For $z \in \mathbb{C} $, we define
$\lfloor z \rfloor_T=\lfloor u+1/2 \rfloor(1+i)+\lfloor v+1/2 \rfloor(1-i)$, where $z=u(1+i)+v(1-i) (u,v \in  \mathbb{R})$.

\begin{rem}
When $z \in 1+(1+i)\mathbb{Z}[i] $, we have $\left\lfloor z \right\rfloor_H \notin (1+i)\mathbb{Z}[i] $, which is necessary for the finiteness of the continued fraction expansion of elements in $\mathbb{Q}[i]$.
\end{rem}

For $z \in \overline{X} $, we define a transformation  $T_H$ on  $\overline{X}$ by
\begin{align*}
&T_H(z):=\begin{cases}\frac{1}{z}-\left\lfloor \frac{1}{z} \right\rfloor_H& \text{if $z\ne 0$},\\
0 & \text{if $z=0$},
\end{cases}\\
\text{and}\\
&a_H(z):=\begin{cases}\left\lfloor \frac{1}{z} \right\rfloor_H& \text{if $z\ne 0$},\\
0 & \text{if $z=0$}.
\end{cases}
\end{align*}

For $z \in X $, we define a transformation  $T_T$ on  $X$ by
\begin{align*}
&T_T(z):=\begin{cases}\frac{1}{z}-\left\lfloor \frac{1}{z} \right\rfloor_T& \text{if $z\ne 0$},\\
0 & \text{if $z=0$},
\end{cases}\\
\text{and}\\
&a_T(z):=\begin{cases}\left\lfloor \frac{1}{z} \right\rfloor_T& \text{if $z\ne 0$},\\
0 & \text{if $z=0$}.
\end{cases}
\end{align*}

We define J. Hurwitz's algorithm  as follows:
\begin{defn}\label{algolH}
Let $\alpha=\alpha_{(1)}^H \in \overline{X}$. 
We define $\{\alpha_{(n)}^H\}_{n\in \mathbb{Z}_{> 0}}$ and $\{a_n^H\}_{n\in \mathbb{Z}_{> 0}}$recursively as follows:\\
If $\alpha_{(n)}^H\ne 0$,
\begin{align*}
&\alpha_{(n+1)}^H:=T_H(\alpha_{(n)}^H),\\
&a_{n}^H:=a_H(\alpha_{(n)}^H).
\end{align*}
If $\alpha_{(n)}^H=0$, terminate.
\end{defn}

Namely, we consider the following continued fraction expansion:
\begin{align*}
\cfrac{1}{a_1^H + \cfrac{1}{a_2^H + \cfrac{1}{a_3^H+\cfrac{1}{\ldots}}}}.
\end{align*}

\begin{rem}
The continued fraction algorithm by J. Hurwitz\cite{JH} used the following transformation $T$ on $\mathbb{C}$.
For $z \in \mathbb{C}$, the transformation is given by
\begin{align*}
T(z) := -\dfrac{1}{z - \left\lfloor z \right\rfloor_H}.
\end{align*}
\end{rem}

We define Tanaka's algorithm  as follows:
\begin{defn}\label{algolT}
Let $\alpha=\alpha_{(1)}^T \in \overline{X}$. 
We define $\{\alpha_{(n)}^T\}_{n\in \mathbb{Z}_{> 0}}$ and $\{a_n^T\}_{n\in \mathbb{Z}_{> 0}}$recursively as follows:\\
If $\alpha_{(n)}^T\ne 0$,
\begin{align*}
&\alpha_{(n+1)}^T:=T_T(\alpha_{(n)}^T),\\
&a_{n}^T:=a_T(\alpha_{(n)}^T).
\end{align*}
If $\alpha_{(n)}^T=0$, terminate.
\end{defn}

J. Hurwitz \cite{JH} constructed what is now referred to as the dual algorithm for the algorithm in Definition \ref{algolH}.
We require some definitions.

We define $D:=\{z\in \mathbb{C} \mid |z|\geq 1\}$.

For $w \in (1+i)\mathbb{Z}[i]$, we define the region $S_w$ as follows:
When $w \in L(i^k(1+i))$ for $0\leq k <4$,
we define 
\begin{align*}
S_w := \{ z \in \mathbb{C} \mid |z - w| < 1, |z - (w - i^k(1 + i))| \geq 1 \}.
\end{align*}
When $w \in L(i^k2)$ for $0\leq k <4$,
we define 
\begin{align*}
S_w := \{ z \in \mathbb{C} \mid |z - w| < 1, |z - (w - i^k(1 + i))| \geq 1, |z - (w - i^k(1 - i))| \geq 1\}.
\end{align*}
We define 
$S_0:=\{z\in \mathbb{C} \mid |z|<1\}$.

Then, we have (see Figure\ref{fig:jh-3-2}) 
\begin{lem}\label{tiling}
\begin{align*}
\bigsqcup_{w\in (1+i)\mathbb{Z}[i]}S_w=\mathbb{C}\backslash (1+(1+i)\mathbb{Z}[i]).
\end{align*}
\end{lem}

For $z \in \mathbb{C} $, when $z \notin \mathbb{Z}[i] $,
 there exists a unique $w$ such that $z \in S_w $.
 In this case, we define $\lfloor z \rfloor_d := w $. Additionally, when $z \in \mathbb{Z}[i] $, we define $\lfloor z \rfloor_d = z $.

For $z\in D$, we define a map  $T_D:D\to D\cup \{\infty\}$ by
\begin{align*}
&T_D(z):=\begin{cases}\dfrac{1}{z-\lfloor z\rfloor_d}& \text{if $z\notin \mathbb{Z}[i]$},\\
\infty & \text{if $z\in \mathbb{Z}[i]$}.
\end{cases}\\
\text{and}\\
&a_D(z):=\lfloor z\rfloor_d.
\end{align*}

We define the dual algorithm of J. Hurwitz's algorithm as follows.
\begin{defn}\label{algolD}
Let $\alpha=\alpha_{(1)}^D \in D$. 
We define $\{\alpha_{(n)}^D\}_{n\in \mathbb{Z}_{> 0}}$ and $\{a_n^D\}_{n\in \mathbb{Z}_{> 0}}$recursively as follows:\\
If $\alpha_{(n)}^D\ne \infty$,
\begin{align*}
&\alpha_{(n+1)}^D:=T_D(\alpha_{(n)}^D),\\
&a_{n}^D:=a_D(\alpha_{(n)}^D).
\end{align*}
If $\alpha_{(n)}^D=\infty$, terminate.
\end{defn}

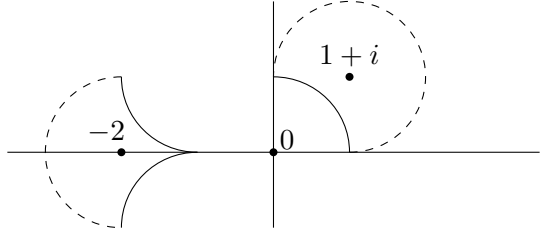
\begin{figure}
    \centering
\begin{tikzpicture}
\draw[dashed] (1,0) arc (-90:180:1);
\draw (-3.5,0)--(3.5,0);
\draw (0,-1)--(0,2);
\draw (1,0) arc (0:90:1);
\draw[dashed] (-2,1) arc (90:270:1);
\draw (-2,1) arc (180:270:1);
\draw (-1,0) arc (90:180:1);

\coordinate (P) at (1,1);
\fill (P) circle [radius=1.5pt];
\draw (1,1)node[above]{$1+i$};
\coordinate (P) at (-2,0);
\fill (P) circle [radius=1.5pt];
\draw (-2-0.2,0)node[above]{$-2$};
\coordinate (O) at (0,0);
\fill (O) circle [radius=1.5pt];
\draw (-0.05,0.17)node[right]{$0$};

 \end{tikzpicture}
    \caption{$S_{-2}$ and $S_{1+i}$. }
    \label{fig:jh-3}
\end{figure}

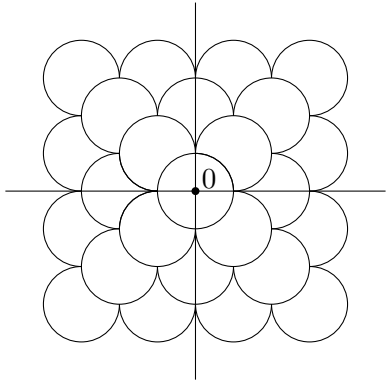
\begin{figure}
    \centering
\begin{tikzpicture}
\draw (0.5,0) arc (-90:180:0.5);
\draw (1,0.5) arc (-90:180:0.5);
\draw (1.5,1) arc (-90:180:0.5);

\draw (-0.5,0) arc (90:360:0.5);
\draw (-1,-0.5) arc (90:360:0.5);
\draw (-1.5,-1) arc (90:360:0.5);

\draw (0,-0.5) arc (180:450:0.5);
\draw (0.5,-1) arc (180:450:0.5);
\draw (1,-1.5) arc (180:450:0.5);

\draw (0,0.5) arc (0:270:0.5);
\draw (-0.5,1) arc (0:270:0.5);
\draw (-1,1.5) arc (0:270:0.5);

\draw (-1.5,1) arc (90:270:0.5);
\draw (-1.5,0) arc (90:270:0.5);

\draw (1,-0.5) arc (270:450:0.5);
\draw (1.5,0) arc (270:450:0.5);
\draw (1.5,-1) arc (270:450:0.5);

\draw (0.5,1) arc (0:180:0.5);
\draw (1,1.5) arc (0:180:0.5);
\draw (0,1.5) arc (0:180:0.5);

\draw (-0.5,-1) arc (180:360:0.5);
\draw (-1,-1.5) arc (180:360:0.5);
\draw (0,-1.5) arc (180:360:0.5);

\draw (-2.5,0)--(2.5,0);
\draw (0,-2.5)--(0,2.5);
\draw (0.5,0) arc (0:90:0.5);
\draw (-1,0.5) arc (90:270:0.5);
\draw (-1,0.5) arc (180:270:0.5);
\draw (-0.5,0) arc (90:180:0.5);

\coordinate (O) at (0,0);
\fill (O) circle [radius=1.5pt];
\draw (-0.05,0.17)node[right]{$0$};

\draw (0.5,0) arc (0:360:0.5);

 \end{tikzpicture}
    \caption{$S_w$ $(w\in (1+i)\mathbb{Z})$ }
    \label{fig:jh-3-2}
\end{figure}

\begin{rem}
J. Hurwitz refers to Definition \ref{algolH} as
 'the first type of continued fraction expansion' and 
 to Definition \ref{algolD} as 'the second type of continued fraction expansion'.
\end{rem}

\begin{example}\label{example1}
The following are examples of the continued fraction expansion based on Algorithm (Definition \ref{algolH}).

\begin{align*}
&\dfrac{2}{5}=[0;2,2],\\
&\dfrac{2+i}{9+8i}=[0;5+i,1-2i],\\
&\sqrt{2+i}-2=[0;\overline{-1-i,-3-i,1+i,3+i}],\\
&1-\sqrt{2}+(-2+\sqrt{2})i=[0;\overline{-1+i,3-3i,1-i,-3+3i}],
\end{align*}
where $\sqrt{2+i}-2=(-0.54465\ldots)+(0.34356\ldots)i$.

The following are examples of the continued fraction expansion based on Algorithm (Definition \ref{algolT}).

\begin{align*}
&\dfrac{2}{5}=[0;2,2],\\
&\dfrac{2+i}{9+8i}=[0;5+i,2-2i,\overline{0}],\\
&\sqrt{2+i}-2=[0;\overline{-1-i,-3-i,1+i,3+i}],\\
&1-\sqrt{2}+(-2+\sqrt{2})i=[0;\overline{-1+i,4-2i,-1+i,-2+4i}].
\end{align*}
\end{example}

\begin{rem}\label{remexample}
Here are some remarks related to Example \ref{example1}:

\begin{enumerate}
\item 
In J. Hurwitz's algorithm, any element of $\mathbb{Q}(i)$ has a finite continued fraction expansion; however, the last partial quotient may sometimes lie in $1+(1+i)\mathbb{Z}[i]$.  
\item 
In Tanaka's algorithm, there are cases where elements of $\mathbb{Q}(i)$ do not have a finite continued fraction expansion. In such cases, the partial quotients may continue with an infinite sequence of zeros, and the continued fraction does not converge.
\item 
Even for an irrational number over $\mathbb{Q}(i)$, there are cases where the continued fractions differ between the two algorithms, as seen in the fourth example.
 \end{enumerate}
\end{rem}

\section{The works of J. Hurwitz}
In this chapter, we present the necessary results from \cite{JH}.

\begin{thm}[J. Hurwitz\cite{JH}]\label{rationalH}
Let $\alpha\in \mathbb{Q}(i)\cap \overline{X}$.
Then, $\{\alpha_{(n)}^H\}$  is a finite sequence.
\end{thm}
We easily obtain the following result. The proof is omitted.
\begin{cor}\label{rationalT}
Let $\alpha\in \mathbb{Q}(i)\cap X$.
Then, $\{\alpha_{(n)}^T\}$ is a finite sequence, or there exists some $m \in \mathbb{Z}_{>0}$ such that for all $n \geq m$, $\alpha_{(n)}^T = -1$.
\end{cor}

The following corresponds to Lagrange's theorem in the case of regular continued fractions.

\begin{thm}[J. Hurwitz\cite{JH}]\label{eperiod}
Let $\alpha \in \overline{X} \backslash \mathbb{Q}(i)$, and 
assume that $\alpha$ is quadratic over $\mathbb{Q}(i)$. Then, $\{\alpha_{(n)}^H\}$ becomes eventually periodic.
\end{thm}

As a corollary, we obtain the following:

\begin{cor}\label{eperiod2}
Let $\alpha \in X \backslash \mathbb{Q}(i)$, and 
assume that $\alpha$ is quadratic over $\mathbb{Q}(i)$. Then, $\{\alpha_{(n)}^T\}$ becomes eventually periodic.
\end{cor}

Let $\alpha$ be a quadratic element over $\mathbb{Q}(i)$, and let $\alpha'$ denote its conjugate over $\mathbb{Q}(i)$.
The following can be regarded as a characterization of purely periodic points.

\begin{thm}[J. Hurwitz\cite{JH}]\label{purlyperiod}
Let $\alpha \in \mathbb{C} \backslash \mathbb{Q}(i)$, and 
assume that $\alpha$ is quadratic over $\mathbb{Q}(i)$ and $\{\alpha_{(n)}^H\}$ becomes purely periodic with
the length $m$.
Then, the sequence
\begin{align*}
(\alpha_{(m)}^H)',(\alpha_{(m-1)}^H)',\ldots,(\alpha_{(1)}^H)',(\alpha_{(m)}^H)',(\alpha_{(m-1)}^H)',\ldots,(\alpha_{(1)}^H)',\ldots
\end{align*}
is generated by the algorithm (Definition \ref{algolD}).
\end{thm}

\begin{rem}\label{jHremark}
We remark that
Theorem \ref{purlyperiod} also states that $|(\alpha_{(m)}^H)'| > 1$.
\end{rem}

\section{The natural extension}
Following H. Nakada \cite{NA}, we describe the natural extension for $T_T$.
 Since the only difference from $T_H$ lies in the boundary of the domain, 
 it  becomes the natural extension of $T_H$. 

We define domains $X_j$ for $1\leq j\leq 4$ as follows (see Figure \ref{fig:jh-4}):
\begin{align*}
X_j:=\{z\in X \mid\text{\ and for all $k$ with $0\leq k\leq 2 $\ }|z-i^{k+j-1}(1+i)/2|>\sqrt{2}/2\}.
\end{align*}
Next, we define domains $X_{j+4}$ for $1\leq j\leq 4$ as follows:
\begin{align*}
X_{j+4}:=\{z\in X\mid\text{\ and for all $k$ with $0\leq k\leq 1$\ }|z-i^{k+j-1}(1-i)/2|<\sqrt{2}/2\}.
\end{align*}
We define domains $W_j$ for $1\leq j\leq 4$ as follows (see Figure \ref{fig:jh-5}):
\begin{align*}
W_j:=\{z\in \mathbb{C} \mid |z|>1,|z-i^{j-1}(1-i)|\geq 1\}.
\end{align*}
Next, we define domains $W_{j+4}$ for $1\leq j\leq 4$ as follows:
\begin{align*}
W_{j+4}:=\{z\in \mathbb{C} \mid |z|>1,\text{\ and for all $k$ with $0\leq k\leq 1 $\ }|z-i^{k+j-1}(1-i)|\geq 1\}.
\end{align*}

We define $\hat{X}$ by:
\begin{align*}
\hat{X}:=\bigcup_{1\leq i\leq 8} (\overline{X_i}\times (\overline{W_i}\cup \{\infty\}).
\end{align*}

We define the transformation $\hat{T_T}$ on $\hat{X}$ as follows:
\begin{align*}
 \hat{T_T}(z,w) := \left(T_T(z), \dfrac{1}{w} - a_T(z)\right), \hspace{1cm} \text{for\ } (z,w) \in \hat{X}.
\end{align*}

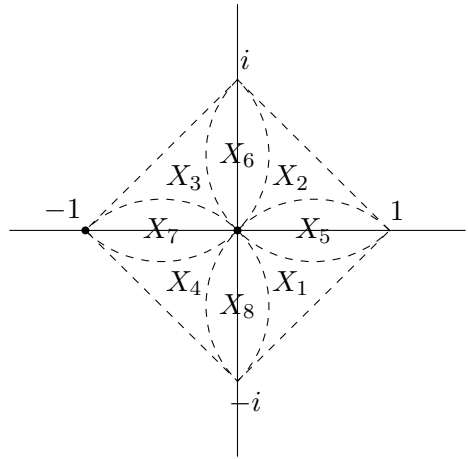
\begin{figure}
    \centering
\begin{tikzpicture}
\draw[dashed] (2,0)--(0,2)--(-2,0)--(0,-2)--(2,0)--cycle;

\draw[dashed] (2,0) arc (45:225:1.4142135);
\draw[dashed] (0,2) arc (135:315:1.4142135);
\draw[dashed] (-2,0) arc (225:405:1.4142135);
\draw[dashed] (0,-2) arc (315:495:1.4142135);

\draw (-3,0)--(3,0);
\draw (0,-3)--(0,3);

\draw (0.7,-0.7)node[]{$X_1$};
\draw (0.7,0.7)node[]{$X_2$};
\draw (-0.7,0.7)node[]{$X_3$};
\draw (-0.7,-0.7)node[]{$X_4$};
\draw (1,0)node[]{$X_5$};
\draw (0,1)node[]{$X_6$};
\draw (-1,0)node[]{$X_7$};
\draw (0,-1)node[]{$X_8$};

\draw (0.1,2)node[above]{$i$};
\draw (0.1,-2)node[below]{$-i$};

\draw (2+0.1,0)node[above]{$1$};
\coordinate (P) at (-2,0);
\fill (P) circle [radius=1.5pt];
\draw (-2-0.3,0)node[above]{$-1$};
\coordinate (O) at (0,0);
\fill (O) circle [radius=1.5pt];

 \end{tikzpicture}
    \caption{$X_j$ ($1\leq j\leq 8$). }
    \label{fig:jh-4}
\end{figure}

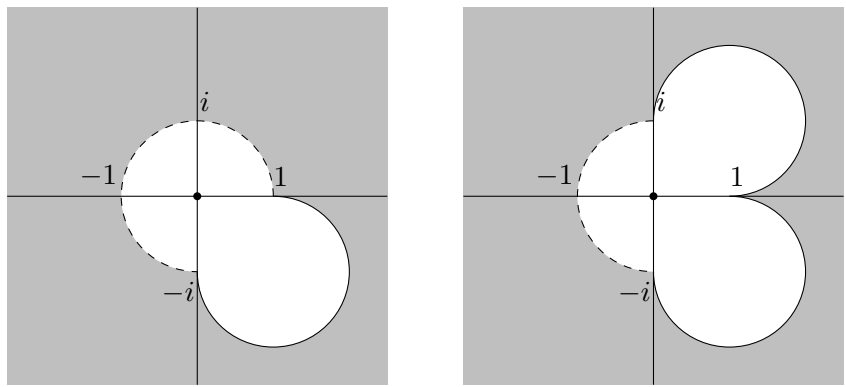
\begin{figure}
    \centering
\begin{tikzpicture}

\fill[lightgray] (-2.5,-2.5) rectangle (2.5,2.5);

\fill[white] (0,0) circle (1);
\fill[white] (1,-1) circle (1);

\draw[dashed] (1,0) arc (0:270:1);
\draw (0,-1) arc (180:450:1);

\draw (0.1,1)node[above]{$i$};
\draw (-0.25,-1)node[below]{$-i$};

\draw (1+0.1,0)node[above]{$1$};
\draw (-1-0.3,0)node[above]{$-1$};
\coordinate (O) at (0,0);
\fill (O) circle [radius=1.5pt];

\draw (-2.5,0)--(2.5,0);
\draw (0,-2.5)--(0,2.5);

\fill[lightgray] (-2.5+6,-2.5) rectangle (2.5+6,2.5);
\fill[white] (0+6,0) circle (1);
\fill[white] (1+6,-1) circle (1);
\fill[white] (1+6,1) circle (1);

\draw (0.1+6,1)node[above]{$i$};
\draw (-0.25+6,-1)node[below]{$-i$};

\draw (1+0.1+6,0)node[above]{$1$};
\draw (-1-0.3+6,0)node[above]{$-1$};
\coordinate (O) at (0+6,0);
\fill (O) circle [radius=1.5pt];

\draw[dashed] (6,1) arc (90:270:1);
\draw (0+6,-1) arc (180:450:1);
\draw (1+6,0) arc (270:540:1);

\draw (-2.5+6,0)--(2.5+6,0);
\draw (0+6,-2.5)--(0+6,2.5);

 \end{tikzpicture}
    \caption{$W_1$ and $W_5$. }
    \label{fig:jh-5}
\end{figure}

Nakada\cite{NA} showed that $\hat{T_T}$ is well-defined and
following theorem holds.
\begin{thm}[Nakada\cite{NA}]\label{nakada}
\begin{enumerate}
\item 
The map $\hat{T_T}$ defined on $\hat{X}$ is bijective except on a set of
Lebesgue measure $0$ and the absolutely continuous invariant probability measure
 $\hat{\mu}$ is given by its density function:
 \begin{align*}
h(z,w)=C_2\cdot \dfrac{1}{|z-w|^4} \hspace{1cm} \text{for\ } (z,w) \in \hat{X}.
\end{align*}
\item 
The map $\hat{T_T}$ is the natural extension of $T_T$. 
\end{enumerate}
\end{thm}

In \cite{NA}, it is shown that $\hat{T_T}$ becomes 
a bijection if we exclude 
$(\bigcup_{1\leq i\leq 8} \partial X_k \times W_k) \cup (\bigcup_{1\leq i\leq 8} X_k \times \partial W_k)$.

Accurately determining purely periodic points requires delicate attention to the boundary of the domain,
  so we have made slight modifications compared to $\hat{X}$.

We define $\widetilde{X}$ by:
\begin{align*}
\widetilde{X}:=\bigcup_{1\leq i\leq 4} (\overline{X_i}\times (W_i\cup \{\infty\})\cup \bigcup_{5\leq i\leq 8} (X_i\times (W_i\cup \{\infty\})).
\end{align*}

We define the transformation $\widetilde{T_H}$ on $\widetilde{X}$ as follows:
\begin{align*} \widetilde{T_H}(z,w) := \left(T_H(z), \dfrac{1}{w} - a_H(z)\right), \hspace{1cm} \text{for\ } (z,w) \in \widetilde{X}. \end{align*}
The fact that $\widetilde{T_H}$ is well-defined can be demonstrated using an argument similar to that for $\hat{T_T}$.
Here, we show that $\widetilde{T_H}$ is injective on the irrational set over $\mathbb{Q}(i)$.
We requires a following lemma.

\begin{lem}\label{For1}
For $1 \leq k \leq 4$, if $\alpha \in L(i^{k-1}(1+i))$,
 the map $z \mapsto \frac{1}{z} - \alpha$ for $z \in \mathbb{C}$ maps $W_k\cup\{\infty\}$ to $S_{-\alpha}$.
For $1 \leq k \leq 4$, if $\alpha \in L(i^{k-1}2)$,
 the map $z \mapsto \frac{1}{z} - \alpha$ for $z \in \mathbb{C}$ maps $W_{k+4}\cup\{\infty\}$ to $S_{-\alpha}$.
\end{lem}
\begin{proof}
Let $\alpha \in L((1+i))$.
We easily obtain
\begin{align*}
(W_1 \cup \{\infty\})^{-1} = \{ z \in \mathbb{C} \mid |z| < 1, |z - (1 + i)| \geq 1 \}.
\end{align*}
Therefore, we have
\begin{align*}
(W_1 \cup \{\infty\})^{-1}-\alpha = \{ z \in \mathbb{C} \mid |z+\alpha| < 1, |z+\alpha - (1 + i)| \geq 1 \}.
\end{align*}
Since $-\alpha\in L(-(1+i))$ holds, we have $S_{-\alpha}=(W_1 \cup \{\infty\})^{-1}-\alpha$.\\
Next, Let $\alpha \in L(2)$.
We easily obtain
\begin{align*}
(W_5 \cup \{\infty\})^{-1} = \{ z \in \mathbb{C} \mid |z| < 1, |z - (1 + i)| \geq 1,|z - (1 - i)| \geq 1 \}.
\end{align*}
Therefore, we have
\begin{align*}
(W_5 \cup \{\infty\})^{-1}-\alpha = \{ z \in \mathbb{C} \mid |z+\alpha| < 1, |z+\alpha - (1 + i)| \geq 1,|z+\alpha - (1 - i)| \geq 1 \}.
\end{align*}
Since $-\alpha\in L(-2)$ holds, we have $S_{-\alpha}=(W_5 \cup \{\infty\})^{-1}-\alpha$.\\
The other cases can be proved similarly.
\end{proof}

\begin{lem}\label{injective}
The map $\widetilde{T_H}$  is injective 
on the set $\widetilde{X}\backslash (\mathbb{Q}(i)\times \mathbb{Q}(i))$.
\end{lem}
\begin{proof}
We assume that $(x_1,y_1), (x_2,y_2)\in \widetilde{X}\backslash (\mathbb{Q}(i)\times \mathbb{Q}(i))$ and
$\widetilde{T_H}(x_1,y_1)=\widetilde{T_H}(x_2,y_2)$.
If $a_H(x_1)=a_H(x_2)$, it is easy to see  that $(x_1,y_1)=(x_2,y_2)$.
We assume that $a_H(x_1)\ne a_H(x_2)$.
We set $(u,v):=\widetilde{T_H}(x_1,y_1)$.
From Lemma \ref{For1}, we have $v\in S_{-a_H(x_1)} \cap S_{-a_H(x_2)}$, which
 contradicts Lemma \ref{tiling}.
\end{proof}

\section{Lemmas}

\begin{figure}
    \centering
\begin{tikzpicture}
\draw(2,0)--(0,2)--(-2,0)--(0,-2)--(2,0)--cycle;

\draw (2,0) arc (45:225:1.4142135);
\draw (0,2) arc (135:315:1.4142135);
\draw (-2,0) arc (225:405:1.4142135);
\draw (0,-2) arc (315:495:1.4142135);

\draw (-2.5,0)--(2.5,0);
\draw (0,-2.5)--(0,2.5);

\draw (1.2,-1.2)node[]{$Y_1$};
\draw (1.2,1.2)node[]{$Y_2$};
\draw (-1.2,1.2)node[]{$Y_3$};
\draw (-1.2,-1.2)node[]{$Y_4$};
\draw (1,-0.2)node[]{$K_1$};
\draw (0+0.2,1)node[]{$K_2$};
\draw (-1,0+0.2)node[]{$K_3$};
\draw (0-0.2,-1)node[]{$K_4$};

\draw (0.1,2)node[above]{$i$};
\draw (0.1,-2)node[below]{$-i$};

\draw (2+0.1,0)node[above]{$1$};
\coordinate (P) at (-2,0);
\fill (P) circle [radius=1.5pt];
\draw (-2-0.3,0)node[above]{$-1$};
\coordinate (O) at (0,0);
\fill (O) circle [radius=1.5pt];

\draw[->] (1,0)--(1,0.35);
\draw[->] (0,1)--(-0.35,1);
\draw[->] (-1,0)--(-1,-0.35);
\draw[->] (0,-1)--(0.35,-1);

 \end{tikzpicture}
    \caption{$K_j$ and $Y_j$ ($1\leq j\leq 4$). }
    \label{fig:jh-6}
\end{figure}
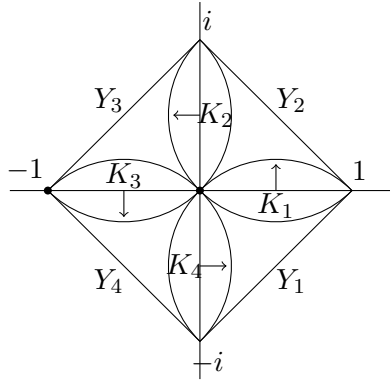

We define sets $K_j$ and $K_j'$ for $1\leq j\leq 4$ as follows:
\begin{align*}
&K_j:=\{z\in C \mid|z-i^{j-1}(1/2-1/2i)|=\dfrac{\sqrt{2}}{2},|z|\leq 1\},\\
&K_j':=\{z\in C\mid |z-i^{j-1}(1/2-1/2i)|=\dfrac{\sqrt{2}}{2},|z|> 1\}.
\end{align*}
We define sets $Y_j$ and $Y_j'$ for $1\leq j\leq 4$ as follows:
\begin{align*}
&Y_j:=\{i^{j-1}(s-it)\mid s,t\in \mathbb{R}, s+t=1,  0\leq s \leq 1\},\\
&Y_j':=\{i^{j-1}(s-it)\mid s,t\in \mathbb{R}, s+t=1,   s\leq 0, t\geq 1\}.
\end{align*}

\begin{lem}[J. Hurwitz\cite{JH}]\label{T_H(Y_1)}
The following holds:
\begin{enumerate}
\item
$T_H(Y_1)=K_4$, $T_H(Y_2)=K_3$, $T_H(Y_3)=K_2$, and $T_H(Y_4)=K_1$. 
\item
$T_H(K_1)=Y_2$, $T_H(K_2)=Y_1$, $T_H(K_3)=Y_4$, and $T_H(K_4)=Y_3$. 
\end{enumerate}

\end{lem}

Similarly, we have:
\begin{lem}\label{T_T(Y_1)}
The following holds:
\begin{enumerate}
\item
$T_T(Y_3)=K_2$, and $T_T(Y_4)=K_1$. 
\item
$T_T(K_1)=Y_4$, $T_T(K_2)=Y_3$, $T_T(K_3)=Y_4$, and $T_T(K_4)=Y_3$. 
\end{enumerate}
\end{lem}

\begin{lem}\label{Letuin}
The following holds:
\begin{enumerate}
\item[(1)]
Let $u\in (K_1\cup K_2)\backslash \mathbb{Q}(i)$.
Then, $T_H(u)\ne T_T(u)$, $T_H^2(u)\ne T_T^2(u)$, $T_H^3(u)=T_T^3(u)$, and $T_H^4(u)=T_T^4(u)\in (K_1\cup K_2)$.
\item[(2)]
Let $u\in (Y_3\cup Y_4)\backslash \mathbb{Q}(i)$.
Then, $T_H(u)=T_T(u)$, $T_H^2(u)\ne T_T^2(u)$, $T_H^3(u)\ne T_T^3(u)$, and $T_H^4(u)=T_T^4(u)\in (Y_3\cup Y_4)$.
\end{enumerate}
\end{lem}
\begin{proof}
We shall prove (1).
Let $u\in K_1$.
Since $K_1^{-1}=Y_2'\cup \{\infty\}$, we have $\frac{1}{u}\in Y_2'$.
Therefore, we have 
$\lfloor \frac{1}{u} \rfloor_H+1+i=\lfloor \frac{1}{u} \rfloor_T$, which implies that
$T_H(u)-(1+i)=T_T(u)$, $T_H(u)\in Y_2$, and $T_T(u)\in Y_4$.  
Therefore, there exists $s>0$ such that
\begin{align*}
&T_H(u)=1+s(-1+i),\\
&T_T(u)=-i+s(-1+i).
\end{align*}
We set $v_j:=T_H^j(u)$ and $w_j:=T_T^j(u)$ 
for $j=1,2,\ldots$.
Since $\lfloor \frac{1}{v_1} \rfloor_H=1-i$ and
$\lfloor \frac{1}{w_1} \rfloor_T=-1+i$,
we have
\begin{align*}
&T_H(v_1)=\dfrac{1}{v_1}-(1-i)=\dfrac{(2s-1)((i-1)s-i)}{2s^2-2s+1},\\
&T_T(w_1)=\dfrac{1}{w_1}+(1-i)=\dfrac{(2s-1)(-(i-1)s-1)}{2s^2-2s+1}.
\end{align*}
We obtain
\begin{align*}
&\dfrac{1}{v_2}=\dfrac{s}{1-2s}+\dfrac{i(s-1)}{1-2s},\\
&\dfrac{1}{w_2}=\dfrac{s-1}{2s-1}+\dfrac{is}{2s-1},
\end{align*}
which implies 
\begin{align*}
\dfrac{1}{w_2}-1-i=\dfrac{1}{v_2}.
\end{align*}
Considering $\frac{1}{v_2}\in Y_4'$ and $\frac{1}{w_2}\in Y_2'$,
we have $T_H(v_2)=T_T(w_2)$.
Since $v_3=w_3\in Y_4$, we see that
$\frac{1}{v_3}=\frac{1}{w_3}\in K_3'$ and 
$\lfloor \frac{1}{v_3} \rfloor_H=\lfloor \frac{1}{w_3} \rfloor_T=-1+i$.
Therefore, we have $v_4=w_4\in K_1$.
In the case where $u \in K_2$, the proof can be done similarly. Moreover, for (2),
since it can be proven in the same way as (1), the proof is omitted.
\end{proof}

\begin{lem}\label{backslash}
Let $u\in X \backslash \mathbb{Q}(i)$.
If $T_H(u)\ne T_T(u)$, then $T_H(u)\in Y_3\cup Y_4$ and  $T_H^4(u)=T_T^4(u)\in  K_1\cup K_2$.
\end{lem}
\begin{proof}
We have 
$\lfloor \frac{1}{u} \rfloor_H\ne\lfloor \frac{1}{u} \rfloor_T$.
Then, we have $\frac{1}{u}\in Y_1'\cup Y_2'$, which implies $T_H(u)\in Y_3\cup Y_4$ and $T_T(u)\in Y_1\cup Y_2$.
Then, similarly to the proof of Lemma \ref{Letuin}, we obtain $T_H^4(u) = T_T^4(u)\in (K_1\cup K_2)$.
\end{proof}

For a quadratic irrational number over $\mathbb{Q}(i)$, certain restrictions
 are imposed on its domain of existence.
We will now present several related lemmas.

\begin{lem}\label{field1}
Let $m, n \in \mathbb{Z}$, $\sqrt{m^2+n^2} \notin \mathbb{Q}$, and $\sqrt{m+ni} \notin \mathbb{Q}(i)$.
 Let  $w \in \mathbb{Q}(\sqrt{m+ni})$ and $w \notin \mathbb{Q}(i)$.
 Then,  $1$, $Re(w)$, and $Im(w)$ are linearly independent over $\mathbb{Q}$, where
$Re(w)$ denotes the real part of $w$, and $Im(w)$ denotes the imaginary part of $w$.
\end{lem}

\begin{proof}
 Since $n \neq 0$, we have $\mathbb{Q}(i) \subset \mathbb{Q}(\sqrt{m+ni})$.
  Let $\alpha = \sqrt{m+ni}$. From the fact $\alpha \notin \mathbb{Q}(i)$, we see that $\mathbb{Q}(\alpha)$ is a quadratic extension over $\mathbb{Q}(i)$.
Let $\beta = \sqrt{m-ni}$ such that $\alpha \beta = \sqrt{m^2+n^2}$.
Then, $\beta$ is the complex conjugate of $\alpha$.  
We will show that $\beta \notin\mathbb{Q}(\alpha)$.
We assume that $\beta \in\mathbb{Q}(\alpha)$.
Then, since $\sqrt{m^2+n^2}=\alpha\beta$, we have $\sqrt{m^2+n^2} \in\mathbb{Q}(\alpha)$.
Therefore, there exist $a+bi,\ c+di\in \mathbb{Q}(i)$$(a,b,c,d\in \mathbb{Q})$ such that
\begin{align}\label{m^2+n^2}
\sqrt{m^2+n^2}=a+bi+(c+di)\alpha.
\end{align}
In  $\mathbb{Q}(\alpha)/\mathbb{Q}(i)$,
let $\sigma$ be the automorphism over $\mathbb{Q}(i)$ such that $\sigma(\alpha) = -\alpha$.
From (\ref{m^2+n^2}), we have
$\sigma(\sqrt{m^2+n^2})=a+bi-(c+di)\alpha$.
If $\sigma(\sqrt{m^2+n^2})=\sqrt{m^2+n^2}$ holds, we get $\sqrt{m^2+n^2}=a+bi$, which 
contradicts $\sqrt{m^2+n^2} \notin \mathbb{Q}$.
Hence, we have $\sigma(\sqrt{m^2+n^2})=-\sqrt{m^2+n^2}$.
We have $\sqrt{m^2+n^2}=(c+di)\alpha$, which implies 
$\beta=c+di$.
Since $\alpha$ is the complex conjugate of $\beta$, 
we have $\alpha=c-di$, which contracts $\alpha \notin \mathbb{Q}(i)$.
Therefore, $\beta \notin\mathbb{Q}(\alpha)$.
Since $[\mathbb{Q}(\alpha,\beta),\mathbb{Q}(\alpha)]=2$ and $[\mathbb{Q}(\alpha),\mathbb{Q}(i)]=2$,
$1, \alpha, \beta, \alpha\beta$ form a basis of $\mathbb{Q}(i, \alpha, \beta)$ over $\mathbb{Q}(i)$.
\\
Let $w\in \mathbb{Q}(\alpha)$ and $w\notin \mathbb{Q}(i)$.
Then, there exist $a+bi,\ c+di\in \mathbb{Q}(i)$$(a,b,c,d\in \mathbb{Q})$ such that
\begin{align}\label{a+bi+}
w=a+bi+(c+di)\alpha.
\end{align}
The complex conjugate of $w$ (denoted by $\overline{w}$) is given by
\begin{align*}
\overline{w} = a - bi + (c - di)\beta.
\end{align*}
Then, we have
\begin{align}
&Re(w)=\dfrac{w+\overline{w}}{2}=a+\dfrac{c+di}{2}\alpha+\dfrac{c-di}{2}\beta,\label{re(w)}\\
&Im(w)=\dfrac{w-\overline{w}}{2i}=b+\dfrac{c+di}{2i}\alpha-\dfrac{c-di}{2i}\beta.\label{im(w)}
\end{align}
We assume that $p\cdot re(w)+q\cdot im(w)+r=0$ for some $p,q,r\in \mathbb{Q}$.
From (\ref{re(w)}) and (\ref{im(w)}), we have
\begin{align*}
pa+qb+r+(p+q/i)\dfrac{c+di}{2}\alpha+(p-q/i)\dfrac{c-di}{2}\beta=0.
\end{align*}
Since $1, \alpha, \beta$ are linearly independent over $\mathbb{Q}(i)$,
we have
\begin{align*}
pa+qb+r=0,\  (p+q/i)\dfrac{c+di}{2}=0,\  (p-q/i)\dfrac{c-di}{2}=0.
\end{align*}
If it holds that $\frac{c+di}{2}=0$ or $\frac{c-di}{2}=0$, then we have $w=a+bi$, which contradicts 
$w\notin \mathbb{Q}(i)$.
Therefore, we have $pa+qb+r=0$, $p+q/i=0$ and $p-q/i=0$, which implies $p=q=r=0$.

\end{proof}

\begin{lem}\label{field2}
 Let $m \in \mathbb{Z}_{>0}$ with $\sqrt{m} \notin \mathbb{Q}$.
  For $w \in \mathbb{Q}(i, \sqrt{m})$ and $w \notin \mathbb{Q}(i)$,
   if there exist $k_j \in \mathbb{Q}$ $(j=1,2,3)$
    such that $k_1 Re(w) + k_2 Im(w) + k_3 = 0$, then for $w'$,
    the conjugate of $w$ over $\mathbb{Q}(i)$, we have $k_1 Re(w') + k_2 Im(w') + k_3 = 0$.
\end{lem}
\begin{proof}
 Let $w=a+bi+(c+di)\sqrt{m}$, where $\ a,b,c,d\in \mathbb{Q}$.
 Then, we have $Re(w)=a+c\sqrt{m},\ Im(w)=b+d\sqrt{m}$.
Therefore, we have $k_1(a+c\sqrt{m})+k_2(b+d\sqrt{m})+k_3=0$, which implies
that $k_1a+k_2b+k_3=0$ and  $k_1c+k_2d=0$.
Since $w'=a+bi-(c+di)\sqrt{m}$,
we have $Re(w')=a-c\sqrt{m},\ Im(w')=b-d\sqrt{m}$.
Then, we have $k_1 Re(w') + k_2 Im(w') + k_3 = k_1a+k_2b+k_3-(k_1c+k_2d)\sqrt{m}=0$. 
\end{proof}

\begin{lem}\label{field3}
 Let $m,n \in \mathbb{Z}$ such that $\sqrt{m^2+n^2}\in \mathbb{Z}_{>0}$ and $\sqrt{m+ni}\notin \mathbb{Q}(i)$.
 Then, there exists $l\in \mathbb{Z}_{>0}$ such that $\sqrt{l}\notin \mathbb{Z}$ and $\mathbb{Q}(i,\sqrt{m+ni})=\mathbb{Q}(i,\sqrt{l})$. 
 \end{lem}
\begin{proof}
We set $k=\sqrt{m^2+n^2}$.
First, we assume that $mn\ne 0$.
Let $l$ be the greatest common divisor of $m$ and $n$. We define $m_1, n_1, k_1 \in \mathbb{Z}_{>0}$ as follows:
\begin{align*}
m = m_1 l, \quad n = n_1 l, \quad k = l k_1.
\end{align*}
Then, we have $m_1^2+n_1^2=k_1^2$.
Since either $m_1$ or $n_1$ is even, we assume $n_1$ to be even.
There exist integers $u$ and $v$ such that 
 $m_1=u^2-v^2, n_1=2uv$.
 Then, we have $m_1+n_1i=(u+vi)^2$.
 We have $\sqrt{m+ni}=\epsilon \sqrt{l}(u+vi)$,
 where $\epsilon\in \{-1,1\}$.
 Therefore, we see that  $\sqrt{l}\notin \mathbb{Z}$ and $\mathbb{Q}(i,\sqrt{m+ni})=\mathbb{Q}(i,\sqrt{l})$.
 Next, we assume that $n=0$.
 Then, we have $\mathbb{Q}(i,\sqrt{m+ni})=\mathbb{Q}(i,\sqrt{|m|})$.
 Subsequently, we assume that $m=0$.
 Then, if $n>0$, then we have $\sqrt{ni}=\sqrt{n}\sqrt{i}=\frac{\sqrt{2n}}{2}+\frac{\sqrt{2n}}{2}i$.
 if $n<0$, then we have $\sqrt{ni}=\sqrt{|n|}\sqrt{-i}=\frac{\sqrt{2|n|}}{2}-\frac{\sqrt{2|n|}}{2}i$.
Therefore, we have $\mathbb{Q}(i,\sqrt{ni})=\mathbb{Q}(i,\sqrt{2|n|})$.
It is clear that $\sqrt{2|n|}\notin \mathbb{Z}$.
\end{proof}

\begin{defn}\label{fieldtype}
Let $K$ be a quadratic field over $\mathbb{Q}(i)$.
Then, there exist integers $m,\ n$ such that $K=\mathbb{Q}(i,\sqrt{m+ni})$ and $\sqrt{m+ni}\notin \mathbb{Q}(i)$.
If $\sqrt{m^2 + n^2} \notin \mathbb{Z}_{>0}$, then we say $K$ is of type A. 
If $\sqrt{m^2 + n^2} \in \mathbb{Z}_{>0}$, then we say $K$ is of type B.
\end{defn}

\begin{rem}\label{remark3}
Definition \ref{fieldtype} is well defined. 
If $K$ is of type B, from Lemma \ref{field3}, we see that there exists
 $k\in \mathbb{Z}_{>0}$ such that $K=\mathbb{Q}(i,\sqrt{k})$, and it is clear that $K$ is not of type A.
\end{rem}

\section{Theorems}
To state the Main Theorem, we first introduce several definitions.

We define sets $N_1$ and $N_2$ by
\begin{align*}
&N_1:=\bigcup_{1\leq i\leq 8} (X_i\times W_i)\cup\bigcup_{1\leq i\leq 4} (Y_i\times Y_i')\cup \bigcup_{1\leq i\leq 4} (K_i\times K_i'),\\
&N_2:=\bigcup_{1\leq i\leq 8} (X_i\times W_i)\cup\bigcup_{3\leq i\leq 4} (Y_i\times Y_i')\cup \bigcup_{1\leq i\leq 2} (K_i\times K_i').
\end{align*}

We give the proof of  Corollary \ref{eperiod2}.
\begin{proof}
If $\alpha_{(n)}^H = \alpha_{(n)}^T$ for all $n \geq 1$, then since $\{\alpha_{(n)}^H\}$ is eventually periodic, 
$\{\alpha_{(n)}^T\}$ must also be eventually  periodic.
We assume that there exists $n > 0$ such that $\alpha_{(n)}^H \ne \alpha_{(n)}^T$ and let $m$ be the smallest such $n$.
Thus, we have $\alpha_{(m)}^H \ne \alpha_{(m)}^T$ and $\alpha_{(m-1)}^H = \alpha_{(m-1)}^T$.
Then, from Lemma \ref{Letuin} and \ref{backslash}, we have
$\alpha_{(m-1+4k)}^H = \alpha_{(m-1+4k)}^T$ for $k=1,2,\ldots$.
Since $\alpha_{(m-1+4k)}^H$ for $k=1, 2, \ldots$ is finite, $\{\alpha_{(n)}^T\}$ is eventually periodic.
\end{proof}

\begin{thm}\label{galois}
For $\alpha \in \overline{X}$, the sequence $\alpha^H_1(=\alpha), \alpha^H_2, \ldots, \alpha^H_n, \ldots$ becomes purely periodic if and only if $\alpha$ is a quadratic irrational over $\mathbb{Q}(i)$ and $(\alpha, \alpha') \in \widetilde{X}$.
\end{thm}

\begin{proof}
Let $\alpha (=\alpha_{(1)}^H) \in \overline{X}$,
 and assume that $\{\alpha_{(n)}^H\}$ is purely periodic.
 Let $m$ be the length of the period.
 We set $k=2m$. 
By a well-known argument, $\alpha$ is at most a quadratic number over $\mathbb{Q}(i)$. By Theorem \ref{rationalH}, $\alpha$ is a quadratic irrational over $\mathbb{Q}(i)$.
By Theorem \ref{purlyperiod}, $(\alpha_{(k)}^H)', (\alpha_{(k-1)}^H)', \ldots, (\alpha_{(1)}^H)', (\alpha_{(k)}^H)', (\alpha_{(k-1)}^H)', \ldots, (\alpha_{(1)}^H)', \ldots$ is generated by the algorithm (Definition \ref{algolD}).
Since $\frac{1}{\alpha_{(1)}^H}-a_1^H=\alpha_{(2)}^H$ holds,
 we have $\frac{1}{(\alpha_{(1)}^H)'}-a_1^H=(\alpha_{(2)}^H)'$, which implies that
\begin{align*}
(\alpha_{(1)}^H)'=\dfrac{1}{(\alpha_{(2)}^H)'+a_1^H}.
\end{align*}
Therefore, we have  $(\alpha_{(2)}^H)'\in S_{-a_1^H}$.
We assume that $\alpha_{(1)}^H\in \overline{X_1}$.
Then, we have $-a_1^H\in L(-(1+i))$.
Therefore, we have
\begin{align*}
(\alpha_{(2)}^H)'+a_1^H\in \{z\in C||z|<1,|z-(1+i)|\geq 1 \}=(W_1\cup \{\infty\})^{-1},
\end{align*}
which implies that $(\alpha_{(1)}^H)'\in W_1$.
Thus, we have $(\alpha, \alpha')\in \widetilde{X}$.
Next, we assume  that $\alpha_{(1)}^H\in X_5$. 
Then, we have $-a_1^H\in L(-2)$.
Therefore, we have
\begin{align*}
(\alpha_{(2)}^H)'+a_1^H\in \{z\in C||z|<1,|z-(1+i)|\geq 1,|z-(1-i)|\geq 1 \}=(W_5\cup \{\infty\})^{-1},
\end{align*}
which implies that $(\alpha_{(1)}^H)'\in W_5$.
Thus, we have $(\alpha, \alpha')\in \widetilde{X}$.
The other cases can be proved similarly.

Let $\alpha (=\alpha_{(1)}) \in \overline{X}$.
We assume that $\alpha$ is a quadratic irrational over $\mathbb{Q}(i)$ and 
 $(\alpha, \alpha')\in \widetilde{X}$.
From Theorem \ref{eperiod}, $\{\alpha_{(n)}^H\}$ is eventually periodic.
That is, there exists some integer $M$ such that for $n \geq M$, $\{\alpha_{(n)}^H\}$ becomes periodic.
We assume $\{\alpha_{(n)}^H\}_{n\geq M}$ is purely periodic with the length $l$.
If $M=1$ holds, then $\{\alpha_{(n)}^H\}$ is purely periodic.
We assume that $M>1$.
Then, We have $\widetilde{T_H}(\alpha_{(M-1)}^H, (\alpha_{(M-1)}^H)')=(\alpha_{(M)}^H, (\alpha_{(M)}^H)')$ and
$\widetilde{T_H}(\alpha_{(M-1+l)}^H, (\alpha_{(M-1+l)}^H)')=(\alpha_{(M)}^H, (\alpha_{(M)}^H)')$.
From Lemma \ref{injective}, we have $((\alpha_{(M-1)}^H, (\alpha_{(M-1)}^H)')=\alpha_{(M-1+l)}^H, (\alpha_{(M-1+l)}^H)')$.
By repeating this argument, we have $((\alpha_{(1)}^H, (\alpha_{(1)}^H)')=(\alpha_{(1+l)}^H, (\alpha_{(1+l)}^H)')$.
\end{proof}

Theorem \ref{galois} can be slightly improved as follows.

\begin{thm}\label{galois2}
For $\alpha \in \overline{X} $, the sequence $\alpha^H_1(=\alpha), \alpha^H_2, \ldots, \alpha^H_n, \ldots$ becomes purely periodic if and only if $\alpha$ is a quadratic irrational over $\mathbb{Q}(i)$ and $(\alpha, \alpha') \in N_1$.
\end{thm}

\begin{proof}
Let $\alpha (=\alpha_{(1)}^H) \in \overline{X}$,
 and assume that $\{\alpha_{(n)}^H\}$ is purely periodic.
Then, from Theorem \ref{galois}, we see that
$\alpha$ is a quadratic irrational over $\mathbb{Q}(i)$ and 
$(\alpha, \alpha')\in \widetilde{X}$.
Let $K=\mathbb{Q}(i,\alpha)$.
If $\alpha\notin \bigcup_{1\leq j\leq 4} Y_j\cup \bigcup_{1\leq j\leq 4} K_j$ holds, 
then we have $(\alpha, \alpha')\in N_1$.
We assume that $\alpha\in \bigcup_{1\leq j\leq 4} Y_j\cup \bigcup_{1\leq j\leq 4} K_j$.
First, we assume that $\alpha \in Y_1$.
If $K$ is of type A, from Lemma \ref{field1}, we have $\alpha \notin Y_1$.
Therefore, $K$ is of type B.
Considering  Lemma \ref{field2} and $\alpha'\in W_1$, we have 
 $\alpha'\in Y_1'$.
 Thus, $(\alpha, \alpha')\in N_1$.
 Next, we assume that $\alpha \in K_1$.
From the fact that $K_1^{-1} = Y_2'$ and by Lemma \ref{field1}, $K$ is of type B. 
Considering  Lemma \ref{field2} and $\alpha'\in W_2\cup W_4$, we have 
 $\alpha'\in K_1'$.
 Thus, $(\alpha, \alpha')\in N_1$.
In the case where $\alpha \in \bigcup_{2 \leq j \leq 4} Y_j \cup \bigcup_{2 \leq j \leq 4} K_j$, we can similarly show 
that $(\alpha, \alpha') \in N_1$.\\
We assume that $\alpha$ is a quadratic irrational over $\mathbb{Q}(i)$ and 
 $(\alpha, \alpha')\in N_1$.
Then, it follows that $(\alpha, \alpha') \in \widetilde{X}$.
  From Theorem \ref{galois2}, we obtain the desired conclusion.
\end{proof}

Regarding Tanaka's algorithm, we have the following.

\begin{thm}\label{galois3}
For $\alpha \in X\backslash \mathbb{Q}(i)$, the sequence $\alpha^T_1(=\alpha), \alpha^T_2, \ldots, \alpha^T_n, \ldots$ becomes purely periodic if and only if $\alpha$ is a quadratic irrational over $\mathbb{Q}(i)$ and $(\alpha, \alpha') \in N_2$.
\end{thm}

\begin{proof}
Let $\alpha (=\alpha_{(1)}^T) \in X$,
 and assume that $\{\alpha_{(n)}^T\}$ is purely periodic.
It is not difficult to see that $\alpha$ is  a quadratic irrational number over $\mathbb{Q}(i)$.
We assume that for all $n\geq 1$  $\alpha_{(n)}^T=\alpha_{(n)}^H$.
Since $\{\alpha_{(n)}^T\}$ is purely periodic, from Theorem \ref{galois2}, 
we have $(\alpha, \alpha') \in N_1$.
We assume that $\alpha\in K_3 \cup K_4$.
Then, from Lemma \ref{T_T(Y_1)}, for all $n\geq 2$, $\alpha\ne \alpha_{(n)}^T$
, which contradicts that $\{\alpha_{(n)}^T\}$ is purely periodic.
Therefore, we have $\alpha\notin K_3 \cup K_4$ , which implies that
$(\alpha, \alpha') \in N_2$.
Next, we assume that there exists $k\in \mathbb{Z}_{>0}$ such that $\alpha_{(k)}^T=\alpha_{(k)}^H$
and $\alpha_{(k+1)}^T\ne\alpha_{(k+1)}^H$.
From Lemma \ref{backslash}, we have 
$\alpha_{(k+4)}^T\in K_1 \cup K_2$.
From Lemma \ref{Letuin} and considering that $\{\alpha_{(n)}^T\}$ is purely periodic, we see that
$\alpha_{(1)}^T\in K_1 \cup K_2$ or $\alpha_{(1)}^T\in Y_3 \cup Y_4$.
From Lemma \ref{Letuin}, we obtain
\begin{align}\label{(4n+1)}
\alpha_{(4n+1)}^T=\alpha_{(4n+1)}^H\ \ \ \text{for all $n\geq 0$.}
\end{align}
Let $g\in \mathbb{Z}_{>0}$ be the length of the period of $\{\alpha_{(n)}^T\}$.
From Lemma \ref{T_T(Y_1)}, we see that $g$ is even.
Therefore, from (\ref{(4n+1)}), we have $\alpha_{(2g+1)}^T=\alpha_{(2g+1)}^H$, which implies
 $\alpha_{(1)}^H=\alpha_{(2g+1)}^H$.
Thus, $\{\alpha_{(n)}^H\}$ is purely periodic.
From Theorem \ref{galois2}, 
we have $(\alpha, \alpha') \in N_1$. 
In the same way as in the previous discussion, we can show that $(\alpha, \alpha') \in N_2$. 

We assume that $\alpha$ is a quadratic irrational over $\mathbb{Q}(i)$ and 
 $(\alpha, \alpha')\in N_2$.
Since $(\alpha, \alpha')\in N_1$ holds, 
$\{\alpha_{(n)}^H\}$ is purely periodic.
If for all $n\geq 1$  $\alpha_{(n)}^T=\alpha_{(n)}^H$ holds, 
$\{\alpha_{(n)}^T\}$ is purely periodic.
Next, we assume that there exists $k\in \mathbb{Z}_{>0}$ such that $\alpha_{(k)}^T=\alpha_{(k)}^H$
and $\alpha_{(k+1)}^T\ne\alpha_{(k+1)}^H$.
From Lemmas \ref{T_H(Y_1)} and \ref{backslash} , we obtain, for all $n > k $,
$\alpha_{(k)}^H \in \bigcup_{1 \leq j \leq 4} Y_j \cup \bigcup_{1 \leq j \leq 4} K_j$.
Therefore, we have $\alpha_{(1)}^H \in \bigcup_{1 \leq j \leq 4} Y_j \cup \bigcup_{1 \leq j \leq 4} K_j$.
Since $(\alpha, \alpha')\in N_2$ holds, we have $\alpha_{(1)}^H \in \bigcup_{3 \leq j \leq 4} Y_j \cup \bigcup_{1 \leq j \leq 2} K_j$.
From Lemma \ref{Letuin}, we have for all $n\geq 0$,
\begin{align}\label{1+4n}
\alpha_{(1+4n)}^H=\alpha_{(1+4n)}^T.
\end{align}
Let $g\in \mathbb{Z}_{>0}$ be the length of the period of $\{\alpha_{(n)}^H\}$.
From Lemma \ref{T_T(Y_1)}, we see that $g$ is a multiple of $4$.
Then, we have  $\alpha_{(1)}^H=\alpha_{(1+g)}^H$.
From (\ref{1+4n}), we have $\alpha_{(1)}^T=\alpha_{(1+g)}^T$.
Thus, $\{\alpha_{(n)}^T\}$ is purely periodic.
\end{proof}

Let $n$ be a non-square integer. It is known that   
$\sqrt{n} - \lfloor \sqrt{n} \rfloor $ has a purely periodic continued fraction expansion.
 As an application of theorems, we will demonstrate that a similar result holds for J. Hurwitz's continued fractions.

\begin{thm}\label{galois4}
Let $m,n$ be integers with $\sqrt{m+ni}\notin \mathbb{Q}(i)$.
Let $\alpha=\sqrt{m+ni}-\lfloor \sqrt{m+ni}\rfloor_H$.
Then, $\{\alpha_{(n)}^H\}$ is purely periodic.
\end{thm}
\begin{proof}

For $m$ and $n$ with $|m| \leq 3$ and $|n| \leq 3$, we have the following Table \ref{t1} and Table \ref{t2}, which demonstrate that the theorem holds within this range.

We assume that $|m|\geq 4$ or $|n|\geq 4$.
Then, we have 
\begin{align}\label{m+ni}
&|-\sqrt{m+ni}-\lfloor \sqrt{m+ni}\rfloor_H|
\geq|-2\sqrt{m+ni}|-|\sqrt{m+ni}-\lfloor \sqrt{m+ni}\rfloor_H|\geq\\\nonumber
&\geq 4-1\geq 3>2\sqrt{2}.
\end{align}
Therefore, we see that for all $1\leq j \leq 8$, $\alpha'\in W_j$.
Since $(\alpha,\alpha')\in \widetilde{X}$ holds, from Theorem \ref{galois},
we have the claim of the theorem.

\begin{table}[H]
\caption{Continued fraction expansion of $\sqrt{m+ni}-\lfloor \sqrt{m+ni}\rfloor_H$ with $|m|,|n|<3$}
\label{t1}
\begin{tabular}{l|l}
$\sqrt{m+ni}-\lfloor \sqrt{m+ni}\rfloor_H$ & the continued fraction expansion \\
\hline
$\sqrt{-2-2i}-(1-i)$&$[0;\overline{-1+i,2-2i}]$\\
\hline
$\sqrt{-2-i}-(-2i)$&$[0;\overline{1-i,-1+3i,-1+i,1-3i}]$\\
\hline
$\sqrt{-2}-2i$&$[0;\overline{2i,4i}]$\\
\hline
$\sqrt{-2+i}-2i$&$[0;\overline{1+i,-1-3i,-1-i,1+3i}]$\\
\hline
$\sqrt{-2+2i}-(1+i)$&$[0;\overline{-1-i,2+2i}]$\\
\hline
$\sqrt{-1-2i}-(1-i)$&$[0;\overline{-2+2i,2-2i}]$\\
\hline
$\sqrt{-1-i}-(1-i)$&$[0;\overline{-2,2-2i}]$\\
\hline
$\sqrt{-1+i}-(1+i)$&$[0;\overline{-2,2+2i}]$\\
\hline
$\sqrt{-1+2i}-(1+i)$&$[0;\overline{-2-2i,2+2i}]$\\
\hline
$\sqrt{-i}-(1-i)$&$[0;\overline{-2-2i,2-2i}]$\\
\hline
$\sqrt{i}-(1+i)$&$[0;\overline{-2+2i,2+2i}]$\\
\hline
$\sqrt{1-2i}-(1-i)$&$[0;\overline{2-2i}]$\\
\hline
$\sqrt{1-i}-(1-i)$&$[0;\overline{-2i,2-2i}]$\\
\hline
$\sqrt{1+i}-(1+i)$&$[0;\overline{2i,2+2i}]$\\
\hline
$\sqrt{1+2i}-(1+i)$&$[0;\overline{2+2i}]$\\
\hline
$\sqrt{2-2i}-(1-i)$&$[0;\overline{1-i,2-2i}]$\\
\hline
$\sqrt{2-i}-2$&$[0;\overline{-1+i,-3+i,1-i,3-i}]$\\
\hline
$\sqrt{2}-2$&$[0;\overline{-2,4}]$\\
\hline
$\sqrt{2+i}-2$&$[0;\overline{-1-i,-3-i,1+i,3+i}]$\\
\hline
$\sqrt{2+2i}-(1+i)$&$[0;\overline{1+i,2+2i}]$\\
\hline
\end{tabular}
\end{table}

\begin{table}[H]
\caption{Continued fraction expansion of $\sqrt{m+ni}-\lfloor \sqrt{m+ni}\rfloor_H$ for $|m|,|n|\leq 3$ with either $|m|=3$ or $|n|=3$}
\label{t2}
\begin{tabular}{l|l}
$\sqrt{m+ni}-\lfloor \sqrt{m+ni}\rfloor_H$ & the continued fraction expansion \\
\hline
$\sqrt{-3-3i}-(-2i)$&$[0;\overline{2,-2,1-i,1-i,-1+i,1-3i}]$\\
\hline
$\sqrt{-3-2i}-(-2i)$&$[0;\overline{2,-1+i,1-3i}]$\\
\hline
$\sqrt{-3-i}-(-2i)$&$[0;\overline{2-2i,-4i}]$\\
\hline
$\sqrt{-3}-2i$&$[0;\overline{4i}]$\\
\hline
$\sqrt{-3+i}-2i$&$[0;\overline{2+2i,4i}]$\\
\hline
$\sqrt{-3+2i}-2i$&$[0;\overline{2,-1-i,1+3i}]$\\
\hline
$\sqrt{-3+3i}-2i$&$[0;\overline{2,-2,1+i,1+i,-1-i,1+3i}]$\\
\hline
$\sqrt{3-3i}-2$&$[0;\overline{2i,2i,-1+i,1-i,1-i,3-i}]$\\
\hline
$\sqrt{3-2i}-2$&$[0;\overline{2i,-1+i,-3+i,-2i,1-i,3-i}]$\\
\hline
$\sqrt{3-i}-2$&$[0;\overline{-2+2i,4}]$\\
\hline
$\sqrt{3}-2$&$[0;\overline{-4,4}]$\\
\hline
$\sqrt{3+i}-2$&$[0;\overline{-2-2i,4}]$\\
\hline
$\sqrt{3+2i}-2$&$[0;\overline{-2i,-1-i,-3-i,2i,1+i,3+i}]$\\
\hline
$\sqrt{3+3i}-2$&$[0;\overline{-2i,-2i,-1-i,1+i,1+i,3+i}]$\\
\hline
$\sqrt{-2-3i}-(1-i)$&$[0;\overline{2i,-1+i,1-i,-2+2i,-2i,1-i,-1+i,2-2i}]$\\
\hline
$\sqrt{-1-3i}-(1-i)$&$[0;\overline{2i,2-2i}]$\\
\hline
$\sqrt{-3i}-(1-i)$&$[0;\overline{2+2i,2-2i}]$\\
\hline
$\sqrt{1-3i}-(1-i)$&$[0;\overline{2,2-2i}]$\\
\hline
$\sqrt{2-3i}-(1-i)$&$[0;\overline{2,-1+i,-1+i,-2+2i,-2,1-i,1-i,2-2i}]$\\
\hline
$\sqrt{-2+3i}-(1+i)$&$[0;\overline{-2i,-1-i,1+i,-2-2i,2i,1+i,-1-i,2+2i}]$\\
\hline
$\sqrt{-1+3i}-(1+i)$&$[0;\overline{-2i,2+2i}]$\\
\hline
$\sqrt{3i}-(1+i)$&$[0;\overline{2-2i,2+2i}]$\\
\hline
$\sqrt{1+3i}-(1+i)$&$[0;\overline{2,2+2i}]$\\
\hline
$\sqrt{2+3i}-(1+i)$&$[0;\overline{2,-1-i,-1-i,-2-2i,-2,1+i,1+i,2+2i}]$\\
\end{tabular}
\end{table}
\end{proof}

\begin{rem}
It is not difficult to see that
for $\alpha\in \overline{X}$, if $\alpha=[0,a_1^H,a_2^H,\ldots,a_n^H,\ldots]$ holds,
then we have $-\alpha=[0,-a_1^H,-a_2^H,\ldots,-a_n^H,\ldots]$.
Therefore, for $a \in \mathbb{Q}(i)$, $\sqrt{a}$ has two possible values as solutions to $X^2 - a$, but whether its continued fraction is purely periodic or not is irrelevant for either solution.
\end{rem}

We have following Theorem.

\begin{thm}\label{galois5}
Let $m,n$ be integers with $\sqrt{m+ni}\notin \mathbb{Q}(i)$.
Let $\alpha=\sqrt{m+ni}-\lfloor \sqrt{m+ni}\rfloor_T$.
Then, $\{\alpha_{(n)}^T\}$ is purely periodic.
\end{thm}

We need following Lemma.

\begin{lem}\label{galois5-2}
Let $m,n$ be integers with $\sqrt{m+ni}\notin \mathbb{Q}(i)$.
Let $a,b$ be integers.
Let $\alpha=\sqrt{m+ni}+a+bi$ 
Then, 
$\alpha\notin \bigcup_{1\leq j\leq 4}K_j$.
\end{lem}

\begin{proof}
We assume that $\alpha\in \bigcup_{1\leq j\leq 4}K_j$.
Then, we have $\alpha^{-1}\in \bigcup_{1\leq j\leq 4}Y_j'$.
We set $K=\mathbb{Q}(i,\sqrt{m+ni})$.
If $K$ is of type A, from Lemma \ref{field1}, we have $\alpha^{-1}\notin \bigcup_{1\leq j\leq 4}Y_j'$.
Therefore, $K$ is of type B.
We assume $\alpha\in K_1$.
First, we assume that $mn\ne 0$.
Then, from the proof Lemma \ref{field3}, we see that
there exist $l\in \mathbb{Z}_{>0}$ and $u,v\in \mathbb{Z}$ with $uv\ne 0$ such that $\sqrt{l}\notin \mathbb{Z}$ and
$\sqrt{m+ni}=\sqrt{l}(u+vi)$.
Then, we have
\begin{align*}
\left(\sqrt{l}u+a-\dfrac{1}{2}\right)^2+\left(\sqrt{l}v+b+\dfrac{1}{2}\right)^2=\dfrac{1}{2},
\end{align*}
which implies 
\begin{align}\label{l(u^2+v^2)}
&l(u^2+v^2)+a^2+b^2-a+b=0,
\end{align}
However, we have easily 
\begin{align*}
l(u^2+v^2)+a^2+b^2-a+b>0,
\end{align*}
which contradicts (\ref{l(u^2+v^2)}).
Next, we assume that $n=0$.
Then, we have
\begin{align*}
\left(\sqrt{m}+a-\dfrac{1}{2}\right)^2+\left(b+\dfrac{1}{2}\right)^2=\dfrac{1}{2},
\end{align*}
which implies that $a-\frac{1}{2}=0$, which contradicts $a\in \mathbb{Z}$.
 Subsequently, we assume that $m=0$.
 we assume $n>0$. 
 Then, we have $\sqrt{ni}=\sqrt{n}\sqrt{i}=\frac{\sqrt{2n}}{2}+\frac{\sqrt{2n}}{2}i$, where
$\sqrt{2n}\notin \mathbb{Z}$.
Then, we have
\begin{align*}
\left(\frac{\sqrt{2n}}{2}+a-\dfrac{1}{2}\right)^2+\left(\frac{\sqrt{2n}}{2}+b+\dfrac{1}{2}\right)^2=\dfrac{1}{2},
\end{align*}
which implies 
\begin{align}\label{n+left}
n+\left(a-\dfrac{1}{2}\right)^2+\left(b+\dfrac{1}{2}\right)^2=\dfrac{1}{2}.
\end{align}
From the fact that $n\geq 1$, the equation (\ref{n+left}) does not hold.
The proof can be done similarly for the case when $n < 0$.
In the case where $\alpha \in \bigcup_{2\leq j\leq 4}K_j$, the proof can be done similarly. 
\end{proof}

We give a proof of Theorem \ref{galois5}.

\begin{proof}
It can be verified that the continued fraction expansion of $\sqrt{m+ni} - \lfloor \sqrt{m+ni} \rfloor_H$ using Algorithm (\ref{algolH}) matches the continued fraction expansion of $\sqrt{m+ni} - \lfloor \sqrt{m+ni} \rfloor_T$ using Algorithm (\ref{algolT}) for
 $m$ and $n$ with $|m| \leq 3 $ and $|n| \leq 3$. Therefore, the theorem holds within this range(see Table \ref{t1}, \ref{t2}).
We assume that $|m|\geq 4$ or $|n|\geq 4$.
Then, we have similarly
\begin{align}\label{m+ni+2}
&|-\sqrt{m+ni}-\lfloor \sqrt{m+ni}\rfloor_T|
\geq|-2\sqrt{m+ni}|-|\sqrt{m+ni}-\lfloor \sqrt{m+ni}\rfloor_T|\geq\\\nonumber
&\geq 4-1\geq 3>2\sqrt{2}.
\end{align}
Therefore, if $\alpha\notin \bigcup_{1\leq j\leq 2} Y_j \cup \bigcup_{1\leq j\leq 4}K_j$ holds, 
then $(\alpha,\alpha')\in N_2$.
From Lemma \ref{galois5-2}, we have 
$\alpha\notin \bigcup_{1\leq j\leq 4}K_j$.
We assume that $\alpha\in Y_3$.
Let $K=\mathbb{Q}(i,\sqrt{m+ni})$.
Then, $K$ is of Type B and from the inequality (\ref{m+ni+2}) and Lemma \ref{field2}, we see that  $\alpha'\in Y_3'$.
Therefore, we have  $(\alpha,\alpha')\in N_2$.
Similarly, if $\alpha\in Y_4$ holds, then we have $(\alpha,\alpha')\in N_2$.
Therefore, from Theorem \ref{galois3}, the clam of Theorem holds.

\end{proof}

\vspace{2cm}

\noindent
Shin-ichi Yasutomi: Faculty of Science, Toho University, 2-1 Miyama, Funabashi Chiba, 274-8510, JAPAN\\
{\it E-mail address: shinichi.yasutomi@sci.toho-u.ac.jp}
\end{document}